\documentclass[12pt]{article}

\usepackage{graphicx}
\usepackage{makeidx}
\usepackage{mathtools,cite,color,xcolor}
\usepackage{fancyhdr,bm}
\usepackage{physics}

\usepackage{amsthm}
\usepackage{tabularx}
\usepackage{mathtools}
\usepackage{txfonts,graphicx,epstopdf,hyperref}
\usepackage[toc,page]{appendix}
\usepackage{caption}
\usepackage{subcaption} 
\usepackage[libertine]{newtxmath}
\usepackage[scr=rsfso]{mathalfa}
\usepackage[mathlines]{lineno}

\theoremstyle{plain}\newtheorem{theorem}{Theorem}[section]
\newtheorem{proposition}{Proposition}[section]
\newtheorem{lemma}{Lemma}[section]

\theoremstyle{definition}

\theoremstyle{remark}
\newtheorem{remark}{Remark}[section]
\theoremstyle{plain}

\numberwithin{equation}{section}

\newcommand{\diag}{\mbox{diag}}

\newcommand{\vecthree} [3]{\left(\begin{matrix}  #1 \\ #2 \\ #3  \end{matrix} \right)}

\begin{document}
\title{
Existence and Asymptotic Stability of Traveling Waves of Two-species Lotka-Volterra Competition-Diffusion Systems via Geometric Singular Perturbations
}
\author{
Chueh-Hsin Chang\footnote{Department of Mathematics, National Chung Cheng University, Min-Hsiung, Chia-Yi 62102, Taiwan.  E-mail: chuehhsin@ccu.edu.tw.
Partially supported by the NSTC
 of Taiwan.}\ \ and
Tzi-Sheng Yang\footnote{Department of Smart Computing and Applied Mathematics, Tunghai
University, Taichung City 407224, Taiwan.  E-mail: tsyang@thu.edu.tw.
Partially supported by the NSTC
 of Taiwan.}
}
\date{}
\maketitle

\begin{abstract}
We consider the two-species Lotka-Volterra competition-diffusion system where one species has a small diffusion rate relative to the other species and has a small competition coefficient. By the geometric singular perturbation theory, we prove the existence of the wavefront connecting the coexistence state to the trivial state, with traveling speed greater than the minimal speed. For such wavefronts, we show that the profiles of both species are monotone. Finally,  we use the geometric singular perturbation theory to estimate the associated Evans function, and give the related stability results.
\medskip

\medskip

\noindent {\bf AMS Subjective Classifications (2000).} 35K40, 35K57, 35B35, 35P15, 35Q92.

\textbf{\ }
\end{abstract}

\newpage
\section{Introduction}
\newcounter{newequation}
The dimensional form of the Lotka-Volterra competition-diffusion system is given by
\begin{subequations}\label{eq:LV-dimensional}
\begin{align}
u_{t}&=d_1 u_{xx}+r_1 u\left( 1-\frac{u}{K_1}-a\frac{v}{K_1}\right) \label{eq:LV-dimensional-u} \\
v_{t}&=d_2 v_{xx}+r_2v \left( 1-\frac{v}{K_2}-b\frac{u}{K_2}\right),\label{eq:LV-dimensional-v}
\end{align}
\end{subequations}
where $d_i$ $[L^2/T]$ are diffusion rates; $r_i$ $[1/T]$ are inherent per-capita growth rates; $K_i$ $[{\rm density}]$ are carrying capacities; and $a$, $b$ are dimensionless competition coefficients. 
By the dimensionless rescaling (see Chapter 3 in Murray \cite{Murray})
\[ u=K_1 U,\ v=K_2 V,\ t=\frac{\tau}{r_1},\ x=\sqrt{\frac{d_1}{r_1}}\xi,\]
and abusing the notations $u$ for $U$, $v$ for $V$, $t$ for $\tau$, and $x$ for $\xi$, \eqref{eq:LV-dimensional} has the dimensionless form
\begin{subequations}\label{eq:LV-0}
\begin{align}
u_{t}&=u_{xx}+u( 1-u-r v) \label{eq:LV-0-u} \\
v_{t}&=dv_{xx}+\ell v( 1-s  u-v) , \label{eq:LV-0-v}
\end{align}
\end{subequations}
where 
\[ d=\frac{d_2}{d_1},\ \ell=\frac{r_2}{r_1},\ r=a\frac{K_2}{K_1},\ s=b\frac{K_1}{K_2}.\]
Hence, $d$ means the relative diffusion rate of $v$-species (with respect to that of $u$-species), $\ell$ means the relative inherent per-capita growth rate (with respect to that of $u$-species). 
The system \eqref{eq:LV-0} has four constant equilibria (when $rs\ne 1$),
\begin{equation}
e_1 :=(0,0),\ e_2  :=(1,0),\ e_3 :=(0,1),\ e_4 :=(u^*,v^*)=\left(\frac{%
1-r}{1-rs },\frac{1-s  }{1-rs  }\right),
\end{equation}
whose local stability is listed  below:
\begin{equation}
\begin{tabular}{ll}
$e_1 $ is unstable; &$e_2  $ is stable if and only if $s >1$; \\
$e_3 $ is stable if and only if $
r >1$; &$e_4 $ is stable if and only if $r ,s  <1$.
\end{tabular}
\label{stability of e1 to e4}
\end{equation}
It was known that the coexistence equilibrium $e_4$ is  positive ($u^*,v^*>0$)  if and only if $r ,s  <1$ or $r ,s >1$, and  that $u^*, v^*\in (0,1)$ for either case. \\

The traveling wave solutions for \eqref{eq:LV-0} have been studied by a lot of work, see, e.g., Conley and Gardner \cite{ConleyGardner1984}, Hung \cite{HLC}, Hosono \cite{Hosono2003}, Kan-on \cite{Kan-on1995,Kan-on1997}, Kan-on and Fang \cite{Kan-onFang}, Kan-on and Yanagida \cite{Kan-onYanagida},  Rodrigo and Mimura \cite{RodrigoMimura2000}, Tang and Fife \cite{TangFife}, and their related reference. The $(e_i$,$e_j)$-wave means a {\em nonnegative}  traveling wavefront connecting $e_i$ and $e_j$. Among those works, the $(e_4$,$e_1)$-waves are less studied than the other $(e_i$,$e_j)$-waves. The existence of  the $(e_4$,$e_1)$-waves was firstly proved in Tang and Fife \cite{TangFife}. Then, the $(e_4$,$e_1)$-waves with exact forms were given in Hung \cite{HLC}, and the existence of non-monotone $(e_4$,$e_1)$-waves was proved in Chen et al. \cite{Chen-Hsiao-Wang}. The stability of $(e_4$,$e_1)$-waves were not studied in Tang and Fife \cite{TangFife}, Hung \cite{HLC}, and Chen et al. \cite{Chen-Hsiao-Wang}. As the two-species system \eqref{eq:LV-0} can be transformed into a monotone system by $(u,v)\to (1-u,v)$, if the transformed $(e_i$,$e_j)$-wave has monotone profiles, its stability can be concluded by the standard theory of monotone systems (see, e.g., Volpert et al. \cite{Vvv}). However, in Tang and Fife \cite{TangFife} the  transformed $(e_4$,$e_1)$-waves did not address that the profiles of their transformed $(e_4$,$e_1)$-waves are monotone or not.  On the other hand, the profiles of the transformed $(e_4$,$e_1)$-waves obtained in Hung \cite{HLC} are surely not monotone. Therefore, the stability of the $(e_4$,$e_1)$-waves in Tang and Fife\cite{TangFife} and Hung \cite{HLC} cannot be concluded by the standard theory of monotone systems.\\

As the point spectrum of traveling waves coincides with the zeros of the associated Evans function, the Evans function approach (see, e.g., Alexander et al. \cite{AGJ}, Evans \cite{Evans3,Evans4}, Kapitula and Promislow \cite{KapitulaPro}) to study the stability of traveling waves was employed in various ways. For example, the orientation index for the proof of stability or instability of traveling waves in Gardner and Jones \cite{GarJones}, Kapitula and Promislow \cite{KapitulaPro},  Pego and Weinstein \cite{PW}, the stability of waves in the gluing bifurcation problem in Nii \cite{Nii}. In general, it is not easy to estimate the zeros of the Evans function without further restrictions on the system. Due to these reasons, in this article, we consider \eqref{eq:LV} where the parameter $d=\varepsilon$ is small and the parameter $s=O(\varepsilon)$. Its biological meaning is that the $v$-species has a small diffusion rate relative to the $u$-species and that the $u$-species has very weak competition. In such a case, it is reasonable that the $v$-species has a chance to survive, and eventually it is possible for the coexistence state $e_4$ to spread in the whole real line. Our system is then written as
\begin{subequations}\label{eq:LV}
\begin{align}
u_{t}&=u_{xx}+u( 1-u-r v) \label{eq:LV-u} \\
v_{t}&=\varepsilon v_{xx}+\ell v( 1-s  u-v) . \label{eq:LV-v}
\end{align}
\end{subequations}
We study the existence of the $(e_4$,$e_1)$-waves of \eqref{eq:LV} as well as their asymptotic stability. \\

In terms of the coordinate $(z,t)$ with  $z= x-c t$ where the constant $c>0$ means the wave speed, system \eqref{eq:LV} becomes
\begin{subequations}\label{eq:LV-zt}
\begin{align}
u_{t}&=u_{zz}+cu_z+u( 1-u-r v) \label{eq:LV-zt-u} \\
v_{t}&=\varepsilon v_{zz}+cv_z+\ell v( 1-s  u-v). \label{eq:LV-zt-v}
\end{align}
\end{subequations}
An $(e_4$,$e_1)$-wave is the nonnegative stationary solution $(u(z),v(z))$ of \eqref{eq:LV-zt} satisfying the profile equation
 \begin{subequations}\label{eq:LV-profile}
\begin{align}
0&=u_{zz}+cu_z+u( 1-u-r v) \label{eq:LV-profile-u} \\
0&=\varepsilon v_{zz}+cv_z+\ell v( 1-s  u-v) , \label{eq:LV-profile-v}
\end{align}
\end{subequations}
and the boundary conditions
\begin{equation}\label{eq:BC}
\lim\limits_{z\to -\infty}(u(z),v(z))=e_4, \ \mbox{ and }\ \lim\limits_{z\to +\infty}(u(z),v(z))=e_1.
\end{equation}
As the eigenvalues of the linearization of  \eqref{eq:LV-profile} at $e_1$ are real only when $c\ge 2$, system \eqref{eq:LV} can only possess $(e_4$,$e_1)$-waves with wave speed $c\ge 2$. We prove the existence of $(e_4$,$e_1)$-waves for $\varepsilon\approx 0$, by using the geometric singular perturbation theory (see Theorem \ref{thm:exist}). Indeed, we find a singular heteroclinic orbit of \eqref{eq:LV-profile} and \eqref{eq:BC} (with $\varepsilon=0$) 
which entirely lies in the 3-dimensional normally hyperbolic singular slow manifold. As $\varepsilon\approx 0$, by Fenichel's theory we show the existence of a heteroclinic orbit of \eqref{eq:LV-profile} and \eqref{eq:BC}
, which entirely lies within the perturbed slow manifold.
Furthermore, we show that for $\varepsilon\approx 0$ both $u$ and $v$ components of the heteroclinic orbit are monotone decreasing  (see Theorem \ref{thm:monotone}). \\

To the authors' knowledge, the stability of $(e_4$,$e_1)$-waves seems to have never been considered before.   We investigate the stability of $(e_4$,$e_1)$-waves of  \eqref{eq:LV} by the spectral analysis. Suppose $L$ is the linearized operator for \eqref{eq:LV} around a $(e_4,e_1)$-wave. According to Alexander et al. \cite{AGJ}, Kapitula and Promislow \cite{KapitulaPro}, Volpert et al. \cite{Vvv}, if the spectrum of $L$ have negative real parts (except that $0$ is a simple eigenvalue, resp.), then the $(e_4,e_1)$-wave is asymptotically stable without shift (with shift, resp.). If on the contrary, some elements of the spectrum of $L$ have positive real parts, then the $(e_4,e_1)$-wave is unstable. Consider the problem $L\varphi =\lambda \varphi$ where $\varphi$ belongs to the space $C_0$, defined by
\begin{equation}\label{def:C0}
C_{0}:=\{\varphi :\mathbb{R}\to \mathbb{C}^2|\ \varphi
\mbox{ is
bounded and uniformly continuous, and }\varphi (z)\to 0\mbox { as }%
|z|\to \infty \}
\end{equation}%
with the supremum norm $\| \varphi \| =\sup_{z\in \mathbb{R}}|\varphi
(z)|$. As a $(e_4,e_1)$-wave is of monostable type, some elements of its essential spectrum must have positive real parts (see Henry \cite{Henry} and Volpert et al. \cite{Vvv}). However, if  $L$ is restricted to act in the weighted space $C_{\omega }$, defined by
\begin{equation}\label{def:C-omega}
C_{\omega }:=\left\{ \varphi \in C_{0}\ | (1+e^{\delta z})\varphi(z) \in
C_{0}\right\}, 
\end{equation}%
for some constant $\delta>0$, then the essential spectrum of $L$ is possible to move to the left half of $\mathbb C$. 
We show  (see Theorem \ref{essential spectrum weighted}) that if  $\ell >2$ and $4<c^2\leq\ell ^2(\ell -1)^{-1}$,  the essential spectrum of $L$ acting $C_{\omega }$ always intersect the right half of $\mathbb C$, 
 whereas if $\ell \le 2$, or $\ell > 2$ with $c^2> \ell ^2(\ell -1)^{-1}$, the essential spectrum of $L$ acting $C_{\omega }$ have negative real parts.  Therefore, for the case  $\ell \le 2$, or $\ell> 2$ with $c^2> \ell ^2(\ell -1)^{-1}$, the stability of the $(e_4,e_1)$-wave in the space $C_{\omega }$ is  determined by  the point spectrum of $L$ (acting $C_{\omega }$). The point spectrum is exactly the set of zeros of the Evans function,  which is the determinant of the fundamental solutions for the eigenvalue problem $L\varphi=\lambda \varphi$ satisfying the zero boundary conditions on $z=+\infty$ and  $z=-\infty$, respectively. Inspired by Alexander et al. \cite{AGJ}, Doelman et al. \cite{DolGarKap,ArjGarKap2001}, Dockery and Lui \cite{DouckeyLui}, Gardner and Jones \cite{GarJones}, van Heijster et al. \cite{vanArjKap2008}, we use the geometric singular perturbation method to approximate the fundamental solutions, which come respectively from the {\em reduced problem} and the {\em layer problem}  of the eigenvalue problem. Owing to the geometric structure of the slow manifold for the reduced eigenvalue problem,
it's turned out that the number of zeros for the Evans function equals that of eigenvalues (counting multiplicities) for some  Sturm-Liouville system. As a result, we prove that the $(e_4,e_1)$-waves are asymptotically stable without shift (see the definition in Volpert et al. \cite{Vvv}) in the space $C_{\omega }$ (see Theorem \ref{thm:stability}). Concluding the discussion, the stability depends on the ranges of $\ell$ and the wave speed $c$, different from the stability of other monostable traveling waves such as those in Chang et al. \cite{CHY2020}, Sattinger \cite{Sattinger} and Volpert et al. \cite{Vvv}. \\

This article is organized as follows. In Section 2 we prove the existence of $(e_4,e_1)$-waves by the geometric singular perturbation method. In Section 3, we prove that the $(e_4,e_1)$-waves obtained in Section 2 are monotone decreasing. In Section 4 we investigate the distribution of the essential spectrum of the $(e_4,e_1)$-waves in the space $C_{\omega }$. In Section 4 we investigate the number of point spectrum with positive real parts and give the stability results of the $(e_4,e_1)$-waves in the space $C_{\omega }$. 

\section{Existence of $(e_4$,$e_1)$-wave}
Assume $r<1$ and $s=O(\varepsilon)$. Using the notation $(\cdot)'=d(\cdot)/dz$ to write \eqref{eq:LV-profile} as the first order system \ref{eq:LV-profile-slow}, \\
\begin{equation}
\left\{
\begin{aligned}\label{eq:LV-profile-slow}
u'&=w_1 \\
w_1'&=-c w_1-u( 1-u-r v)\\ 
v'& =w_2\\ 
\varepsilon w_2' &=-c w_2-\ell v( 1-s  u-v), 
\end{aligned}
\right.
\tag*{$(S_\varepsilon)$}
\end{equation}
which is referred to the {\em slow system}.  Introducing the fast variable $\xi:=z/\varepsilon$, and using $\dot{(~)}=d(~)/d\xi$, \ref{eq:LV-profile-slow} is equivalent to  the {\em fast system} \ref{eq:LV-profile-fast},
\begin{equation}
\left\{
\begin{aligned}\label{eq:LV-profile-fast}
\dot u&=\varepsilon w_1 \\
\dot w_1&=\varepsilon[-c w_1-u( 1-u-r v)]\\ 
\dot v& =\varepsilon w_2\\ 
\dot w_2 &=-c w_2-\ell v( 1-s  u-v). 
\end{aligned}
\right.
\tag*{$(F_\varepsilon)$}
\end{equation}

\noindent In the context of $(e_4$,$e_1)$-wave, we denote the trivial equilibrium and the positive equilibrium of \ref{eq:LV-profile-slow}  by
\[ e_1=(0,0,0,0),\ \   e_{4,\varepsilon}=(u^*_\varepsilon,0, v^*_\varepsilon,0),\]
where $u_\varepsilon^*$, $v_\varepsilon^*$ are dependent of $\varepsilon$, given by
 \[   u^*_\varepsilon=\frac{1-r}{1-r s}=1-r+O(\varepsilon),\ \ \ 
 v^*_\varepsilon=\frac{1-s}{1-r s}=1+O(\varepsilon).\]


\begin{lemma}\label{lem:local-anal-fast-syst}
  For $\varepsilon\approx 0$,  \ref{eq:LV-profile-slow}  can only have  $(e_4,e_1)$-waves with speed $c\ge 2$.   
\end{lemma}
\begin{proof}
The linearization of \ref{eq:LV-profile-slow} is given by
\begin{align}
 J_{\varepsilon}:=\varepsilon^{-1}\left(
\begin{matrix}
0 & \varepsilon &0 &0 \\  \varepsilon(2u-1+rv)  & -\varepsilon c & \varepsilon r u & 0\\
0 &  0  & 0&  \varepsilon  \\ \ell s v & 0 & \ell (2v-1+su) &-c \\
\end{matrix}
\right).
\end{align}
In particular, the linearization  of \ref{eq:LV-profile-slow}  at $e_1$ is given by
\begin{align}
J_{1,\varepsilon}:=\varepsilon^{-1}\left(
\begin{matrix}
0 & \varepsilon &0 &0 \\ -\varepsilon  & -\varepsilon c & 0 & 0 \\
0 &  0  & 0&  \varepsilon  \\ 0 & 0 & -\ell  &-c \\
\end{matrix}
\right),
\end{align}
where the eigenvalues of the left upper $2\times 2$ block are real  if and only if $c\ge 2$. The lower right block has two negative eigenvalues for $\varepsilon\approx 0$.  
  As $(e_4$,$e_1)$-waves need to be nonnegative, the complex eigenvalues at $e_1$ are not allowed, which means their speed $c\ge 2$.
\end{proof}
In the rest of this article, we assume $c>2$.
\smallskip 
Taking $\varepsilon=0$ for \ref{eq:LV-profile-slow} leads to the {\em reduced problem} \ref{eq:LV-profile-slow-0}, 
\begin{equation}
\left\{
\begin{aligned}\label{eq:LV-profile-slow-0}
u'&=w_1 \\
w_1'&=-c w_1-u( 1-u-r v)\\ 
v'& =w_2\\ 
0 &=-c w_2-\ell v( 1-v), 
\end{aligned}
\right.
\tag*{$(S_0)$}
\end{equation}
and taking $\varepsilon=0$ for \ref{eq:LV-profile-fast} leads to the {\em layer problem} \ref{eq:LV-profile-fast-0}, 
\begin{equation}
\left\{
\begin{aligned}\label{eq:LV-profile-fast-0}
\dot u&=0 \\
\dot w_1&=0\\ 
\dot v& =0\\ 
\dot w_2 &=-c w_2-\ell v( 1-v). 
\end{aligned}
\right.
\tag*{$(F_0)$}
\end{equation}
 The reduced problem \ref{eq:LV-profile-slow-0} has  a 3-dimensional critical manifold $\mathscr S $, given by
\begin{align}\label{eq:LV-profile-slow-0-mfld}
\mathscr S :=\big\{(u,w_1,v, w_2)|\ w_2=h(v):=c^{-1}  \ell v( v-1),\  u,w_1, v\in\mathbb R\big\},
\end{align}
which is also  the equilibrium of \ref{eq:LV-profile-fast-0}. The linearization of \ref{eq:LV-profile-fast-0} around $(u,w_1,v,w_2)\in\mathscr S$ is given by
\begin{equation*}
\left(
\begin{array}{ccccc}
0 & 0 & 0 & 0  \\
0 & 0 & 0 & 0  \\
0 & 0 & 0 & 0   \\
0 & 0  & \ell(2v-1)  & -c%
\end{array}%
\right), 
\end{equation*}%
which has the eigenvalues $0$ with multiplicity three (corresponding to the tangential space of $\mathscr S$), and $-c<0$. Hence  $\mathscr S$ is normally hyperbolic. 
\noindent  The flow of \ref{eq:LV-profile-slow-0} restricted on $\mathscr S $  is governed by the equation for $(u,v,w_1)$, given by
\begin{equation}\label{eq:dyna-slow-0}
\vecthree{u'}{w_1'}{v'}=\vecthree{w_1}{-c w_1-u( 1-u-r v)}{(\ell/c) v(v-1)}:=f(u,w_1,v),
\end{equation}
which has the equilibria $e'_1:=(0,0,0)$ and $e'_4:=(1-r,0,0)$. For the flow of \eqref{eq:dyna-slow-0}, let $\mathscr W^s(e'_1)$ denote the stable manifold of $e'_1$ and $\mathscr W^u(e'_4)$ denote the unstable manifold of $e'_4$.

\begin{lemma}\label{lem:dim of W^s(e1') and W^u(e4')}
Assume $c>2$.  For the flow of  \eqref{eq:dyna-slow-0},  $\mathscr W^s(e'_1)$  is 3-dimensional, and  $\mathscr W^u(e'_4)$ is 2-dimensional.
\end{lemma}
\begin{proof}
The linearization of the vector field $f$ for \eqref{eq:dyna-slow-0} is given by 
\[K:=\left(\begin{matrix}
0 &1 &0 \\ 2u-1+rv & -c & ru \\ 0 & 0& (\ell/c)(2v-1) \end{matrix}
\right).
\]
In particular, the linearization of $f$ at $e'_1=(0,0,0)$, given by 
\[K_1:=\left(\begin{matrix}
0 &1 &0 \\ -1 & -c & 0 \\ 0 & 0& -(\ell/c) \end{matrix}
\right),
\]
has the three negative eigenvalues 
\[ -\frac{\ell}{c}<0,\ \frac{1}{2}\left(-c\pm(c^2-4)^{1/2}\right).\]
On the other hand, the linearization of $f$ at  $e_4'=(1-r,0,1)$, given by 
\[K_4:=\left(\begin{matrix}
0 &1 &0 \\ 1-r &  -c & r(1-r) \\ 0 & 0& \ell/c \end{matrix}
\right),
\]
has the two positive and one negative eigenvalues
\[ \nu^+:=\frac{\ell}{c}>0,\ \mu^\pm:=\frac{1}{2}\left(-c\pm\left(c^2+4(1-r)\right)^{1/2}\right).\]
Therefore, for the flow of \eqref{eq:dyna-slow-0}, the stable manifold $\mathscr W^s(e'_1)$ of $e'_1$ is 3-dimensional, while  the unstable manifold $\mathscr W^u(e'_4)$ of $e'_4$ is  2-dimensional. 
\end{proof}
\medskip
We technically define the set $U$ (see Figure \ref{fig:U}) in the $uw_1 v$ space by 
\[ U:=\left\{ (u,w_1,v):\ u,v\in (0,1),\ -u<w_1<0,\ 1-u-rv>0 \right\}.\]
\begin{figure}[h]
\centering\includegraphics[scale=0.6]{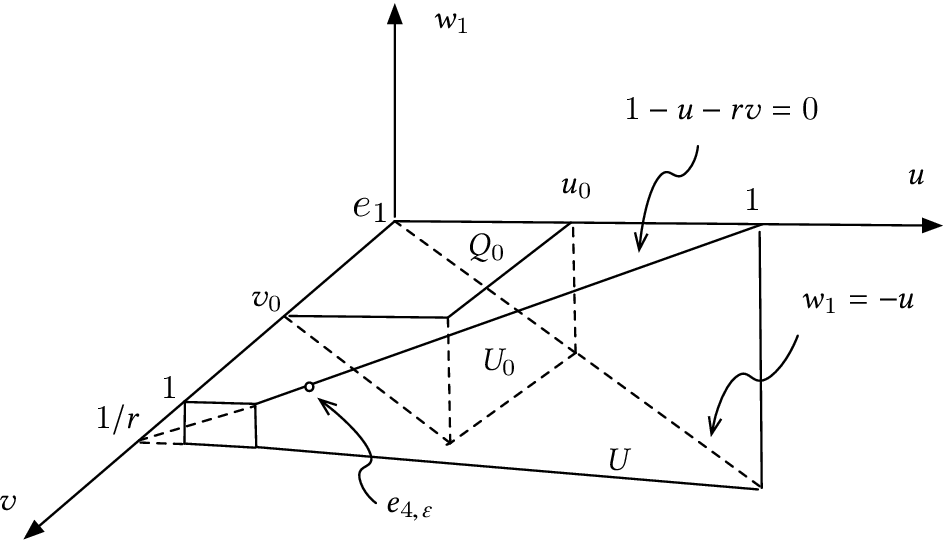}
\caption{}
\label{fig:U}
\end{figure}
\begin{lemma} \label{lem:unstable-direct-e4}
$\mathscr W^u(e'_4)$  intersects $U$. 
\end{lemma}
\begin{proof} Let $\varphi=(-1,-\mu^+,0)$ be the eigenvector of $e'_4$ for the positive eigenvalues $\mu^+$.  Since $1-u^*- r v^*=0$, $e'_4$ lies on  the boundary of $U$ (see Figure \ref{fig:U}), it suffices to show that  $e'_4+\delta \varphi$ lies in $U$ for $\delta\gtrsim 0$. Indeed, 
the $(u, v, w_1)$  components of $e'_4+\delta \varphi$ are given by
\begin{align*}
u&=u^*-\delta\\
w_1&=-\mu^+ \delta\\\
v&=v^*.
\end{align*}
Obviously,  for $\delta\gtrsim 0$ we have $w_1<0$, $u,v\in (0,1)$ because of  $u^*$, $v^*\in (0,1)$, and $1-u-rv=\delta>0$, which means $e'_4+\delta \varphi\in U$.
This completes the proof.
\end{proof}
\begin{lemma}\label{lemma:critical-slow-orbit}
 For the flow of \eqref{eq:dyna-slow-0},  $\mathscr W^s(e'_1)$ contains $U$.
\end{lemma}

\begin{proof} 
  Let $p^0=(u^0,w^0_1,v^0)\in U$. We are going to show that the orbit of \eqref{eq:dyna-slow-0} starting at $p^0$ tends to $e'_1$ as $z\to +\infty$. Let $U_0\subset U$ (see Figure \ref{fig:U}) be  defined by
\[ U_0:=\big\{(u,w_1,v)|\  0< u\le u^0, 0< v\le v^0, -u< w_1<0\big\}.\]
First we show that $U_0$ is positively invariant under the vector field $f$ of \eqref{eq:dyna-slow-0}. 
As $v^0\in (0,1)$, the equation for $v$ in \eqref{eq:dyna-slow-0} implies that the $v$-component of any trajectory starting in $U_0$  remains in  $U_0$. Hence, we only need to show that the set $Q_0$ in $uw_1$-space defined by
\[ Q_0:=\big\{(u,w_1)|\ 0< u\le u^0,  -u<w_1< 0\big\}\]
 is positively invariant under the vector field  $f_1(u,w):=(w_1, -cw_1-u(1-u-rv))$ for all $0<v\le v_0$.  To this end,
it suffices to show that the $f_1(q)\cdot n(q)> 0$ for any $q\in \partial Q_0$ with inward normal $n(q)$. The boundary $\partial Q_0$ is consisted of three subsets, given by
\begin{align*}
A_1&=\big\{(u,w_1)|\ 0\le u\le u^0,\   w_1= 0\big\} \\
 A_2&=\big\{(u,w_1)|\ u=u_0,\  -u_0< w_1<0\big\} \\
A_3&=\big\{(u,w_1)|\ 0< u\le u^0,\  -u= w_1\big\}.
\end{align*}
For $p\in A_1$, $n(q)=(0,-1)$, and  $f_1(q)\cdot n(q)= u(1-u-rv)\ge u(1-u^0-rv^0)> 0$; \\
For $p\in A_2$, $n(q)=(-1,0)$, and  $f_1(q)\cdot n(q)= -w_1>0$; \\
For $p\in A_3$, $n(q)=(1,1)$, and  $f_1(q)\cdot n(q)=w_1-cw_1-u(1-u-rv)=u[(c-2)+u+rv]> u(c-2)>0$, because of  $c>2$.\\

\noindent  Let $(u(z),w_1(z), v(z),w_2(z))$ be the trajectory of \ref{eq:LV-profile-slow-0} starting at $p^0\in U_0$, where $(u(z),w_1(z), v(z))$ is the trajectory of \eqref{eq:dyna-slow-0} starting at $(u^0,w^0_1,v^0)\in U_0$. As $(u(z),w_1(z), v(z))$ remains in $U_0$, we have
\[ v(z)\downarrow 0\  \mbox{and}\  u(z)\downarrow 0,\  \mbox{ as } z\to +\infty.\]
In view that  $-u(z) \le w_1(z)\le 0$ in $U_0$, we have $w_1(z)\to 0$ as  $z\to +\infty$. This completes the proof.
\end{proof}

\begin{lemma} \label{lem:Wu4-intersect-Ns1}
For the flow of  \eqref{eq:dyna-slow-0}, $\mathscr W^u(e'_4)$ intersects $\mathscr W^s(e'_1)$  transversally.
\end{lemma}
\begin{proof}
Since $\mathscr W^s(e'_1)$ contains $U$ by Lemma \ref{lemma:critical-slow-orbit} , and $\mathscr W^u(e'_4)$ intersects $U$ by Lemma \ref{lem:unstable-direct-e4}, we have $\mathscr W^u(e'_4)$ intersects $\mathscr W^s(e'_1)$. Denote by $T_p(\cdot)$ the tangent space of the manifold $(\cdot)$ at $p$.
For  $p\in \mathscr W^u(e'_4) \cap \mathscr W^s(e'_1)$, it's clear that $T_p(\mathscr W^s(e'_1))= \mathbb R^3$ by Lemma \ref{lem:dim of W^s(e1') and W^u(e4')}, which implies
\[ T_p(\mathbb R^3) =T_p(\mathscr  \mathscr W^s(e'_1))+T_p(\mathscr W^u(e'_4)) .\]
Hence, $\mathscr W^u(e'_4)$ intersects $\mathscr W^s(e'_1)$  transversally.
\end{proof}

\begin{theorem}\label{thm:exist}
Assume that $r<1$ and $s=O(\varepsilon)$. For  $\varepsilon\approx 0$, \eqref{eq:LV} has $ (e_4,e_1)$-waves with speed $c$ for all $c>2$.
\end{theorem}
\begin{proof} 
According to Fenichel \cite{Feni},  for $\varepsilon\approx 0$,  \ref{eq:LV-profile-slow} has a normally hyperbolic invariant manifold $\mathscr S_\varepsilon$ given by
\begin{align}\label{eq:LV-profile-slow-mfld}
\mathscr S_\varepsilon :=\big\{(u,w_1,v, w_2)|\ w_2=h_\varepsilon(u,w_1,v),\  u,w_1,v \in\mathbb R\big\},\ \ \mbox{with}\ h_\varepsilon(u,w_1,v)=h(v)+O(\varepsilon),
\end{align}
which contains $e_1$ and $e_{4,\varepsilon}$. i.e., 
\[h_\varepsilon(0,0,0)=h_\varepsilon(u^*_\varepsilon,0,v^*_\varepsilon)=0.\] 
Like the reduced problem \ref{eq:LV-profile-slow-0},  the dynamics of  \ref{eq:LV-profile-slow} restricted  on $\mathscr S_\varepsilon$  is governed by the equations for $(u,w_1,v)$, given by
\begin{align}\label{eq:dyna-slow}
\vecthree{u'}{w_1'}{v'}=f_\varepsilon(u,w_1,v):=f(u,w_1,v)+O(\varepsilon).
\end{align}
For the flow of \eqref{eq:dyna-slow}, let $\mathscr W^s_\varepsilon(e'_1)$ denote the stable manifold of $e'_1$ and  $\mathscr W^u_\varepsilon(e'_{4,\varepsilon})$ denote the unstable manifold of $e'_{4,\varepsilon}=(u^*_\varepsilon,0,v^*_\varepsilon)$. As $\mathscr W^u(e'_4)$ intersects $\mathscr W^s(e'_1)$  transversally by Lemma \ref{lem:Wu4-intersect-Ns1}, $\mathscr W^u_\varepsilon(e'_{4,\varepsilon})$ will intersects $\mathscr W^s_\varepsilon(e'_1)$ transversally for $\varepsilon\approx 0$ (see Szmolyan \cite{Sz}, Theorem 2.2), which means \eqref{eq:dyna-slow} possesses a heteroclinic orbit connecting $e'_{4,\varepsilon}$ and $e'_1$. It follows that \ref{eq:LV-profile-slow} has a heteroclinic  orbit connecting $e_{4,\varepsilon}$ and $e_1$.   
\end{proof}

\section{Monotonicity and uniqueness}
Throughout this section, let \  $\psi_0(z)=(u_0(z),w_{1,0}(z), v_0(z),w_{2,0}(z))$ denote the singular heteroclinic orbit connecting $e_4$ and $e_1$ for the reduced problem \ref{eq:LV-profile-slow-0} constructed in Theorem \ref{thm:exist}. Let \    $\psi_\varepsilon(z)=(u_\varepsilon(z),w_{1,\varepsilon}(z),v_\varepsilon(z),w_{2,\varepsilon}(z))$ denote the heteroclinic orbit connecting $e_{4,\varepsilon}$ and $e_1$ obtained in Theorem \ref{thm:exist}, which is $O(\varepsilon)$-closed to $\psi_0(z)$. In view of Lemma \ref{lem:unstable-direct-e4}-\ref{lem:Wu4-intersect-Ns1} we see that $w_{2,0}(z)=h(v_0(z))$ with $(u_0(z),w_{1,0}(z),v_0(z))\in U$, for all $z\in\mathbb R$. This implies that
\[ w_{1,0}(z)<0,\ \  w_{2,0}(z)=c^{-1}\ell v_0(z)(v_0(z)-1)<0,\ \ \forall\ z\in\mathbb R.\] 
Let $\delta>0$ be small. There exist $z^-<z^+$ such that 
\[ |\psi_0(z)-e_1|<\delta,\ \forall\ z>z^+,\ \mbox{and}\ \  |\psi_0(z)-e_4|<\delta,\ \forall\ z<z^-.\]
Since $\psi_\varepsilon(z)=\psi_0(z)+O(\varepsilon)$ and $[z^-,z^+]$ is compact, it's clear that there exist  $\varepsilon_0>0$ such that for $\varepsilon<\varepsilon_0$, 
\begin{equation}\label{eq:w_1,w_2<0 in cpt interval}
w_{1,\varepsilon}(z)<0,\ w_{2,\varepsilon}(z)<0,\ \forall z\in [z^-,z^+]
\end{equation}
and 
\begin{align}\label{eq:psi-varepsilon-asymp}
|\psi_\varepsilon(z)-e_1|<\delta,\ \forall\ z>z^+,\ \mbox{and}\ \  |\psi_\varepsilon(z)-e_{4,\varepsilon}|<\delta,\ \forall\ z<z^-.
\end{align}
Hence, both $u_\varepsilon(z)$ and $v_\varepsilon(z)$ are monotone decreasing  for all $z \in [z^-,z^+]$. In the following, we will show that both $u_\varepsilon(z)$ and $v_\varepsilon(z)$ are monotone decreasing  for all $z \in (-\infty,z^-)\cup (z^+,+\infty)$.
\begin{lemma}\label{lem:decreasing z>z+}
 Both $u_\varepsilon(z)$ and $v_\varepsilon(z)$ are monotone decreasing  for all $z> z^+$.  
\end{lemma}
\begin{proof} 
We show that both $w_{1,\varepsilon}(z)<0$ and $w_{2,\varepsilon}(z)<0$ for all $z>z^+$ by contradiction. 
If  $w_{1,\varepsilon}(z)>0$ for some $z>z^+$, there must be a $\eta>z^+$ where $w_{1,\varepsilon}(\eta)=0$  and $w'_{1,\varepsilon}(\eta)\geq 0$. In this case, by \ref{eq:LV-profile-slow} and \eqref{eq:psi-varepsilon-asymp}, we have 
\[ w'_{1,\varepsilon}(\eta)=-u_\varepsilon(\eta)(1-u_\varepsilon(\eta)-rv_\varepsilon(\eta))<-u_\varepsilon(\eta)(1-\delta-r\delta)<0,\]
which contradicts to $w'_{1,\varepsilon}(\eta)\geq 0$. \\
Similarly, if  $w_{2,\varepsilon}(z)>0$ for some $z>z^+$, there must be a $\zeta>z^+$ where $w_{2,\varepsilon}(\zeta)=0$  and $w'_{2,\varepsilon}(\zeta)\geq 0$. In this case,  by \ref{eq:LV-profile-slow} and \eqref{eq:psi-varepsilon-asymp}, we have 
\[ w'_{2,\varepsilon}(\zeta)=-v_\varepsilon(\zeta)(1-v_\varepsilon(\zeta)-su_\varepsilon(\zeta))<-v_\varepsilon(\zeta)(1-\delta-s\delta)<0,\] 
which contradicts to $w'_{2,\varepsilon}(\zeta)\geq 0$. Therefore, $w_{2,\varepsilon}(z)<0$ for all $z>z^+$.  Thus both $w_{1,\varepsilon}(z)<0$ and $w_{2,\varepsilon}(z)<0$ for all $z>z^+$. 
\end{proof}

\begin{lemma}
For $\varepsilon\approx 0$, \ref{eq:LV-profile-slow}  has two distinct positive eigenvalues $0<\mu^+_{1,\varepsilon}<\mu^+_{2,\varepsilon}$ at $e_{4,\varepsilon}$. Up to a constant multiple, the eigenvalue $\mu^+_{2,\varepsilon}$ is associated an eigenvector with the signs of  $(u,w_1,v,w_2)$ given by $(-,-,-,-)$, whereas  the eigenvalue $\mu^+_{1,\varepsilon}$ is associated  an eigenvector with the signs of  $(u,w_1,v,w_2)$ given by $(-,-,+,+)$.    
\end{lemma}

\begin{proof}
The  linearization of \ref{eq:LV-profile-slow}  at $e_{4,\varepsilon}$ is given by
\begin{align}
J_{4,\varepsilon}:=\varepsilon^{-1}\left(
\begin{matrix}
0 & \varepsilon &0 &0 \\ \varepsilon u^*_\varepsilon   & -\varepsilon c  & \varepsilon r u^*_\varepsilon & 0\\
0 &  0  & 0&  \varepsilon  \\ \ell s v^*_\varepsilon & 0 & \ell v^*_\varepsilon  &-c \\
\end{matrix}
\right),
\end{align}
whose  eigenvalues $\mu$ satisfy 
\[ F(\mu):=p_\varepsilon(\mu)q_\varepsilon(\mu) - \ell r s u^*_\varepsilon v^*_\varepsilon=0,\] where
\begin{align}
 p_\varepsilon(\mu)=\mu(\varepsilon\mu+c)- \ell v^*_\varepsilon,\ \  \ q_\varepsilon(\mu)=\mu( c+\mu)- u^*_\varepsilon.
\end{align}
For $\varepsilon \approx 0$, $p_\varepsilon(\mu)$ has one positive zero $\alpha^+_{\varepsilon}$, given by
\begin{align}\label{eq:alpha-pm}
\alpha^+_{\varepsilon}
=\frac{\ell}{c}+O(\varepsilon),\ 
\end{align}
 while  $q_\varepsilon(\mu)$  has  one positive zero $\beta^+_{\varepsilon}$, denoted by
\begin{align}\label{eq:mu-pm-q}
\beta^+_{\varepsilon}= \frac{1}{2}[-c+(c^2+4u^*_\varepsilon)^{1/2} ]=\frac{1}{2}[-c+ (c^2+4(1-r))^{1/2} ]+O(\varepsilon).
\end{align}
Let $\gamma_{1,\varepsilon}=\min(\alpha^+_{\varepsilon},\beta^+_{\varepsilon})$ and $\gamma_{2,\varepsilon}=\max(\alpha^+_{\varepsilon},\beta^+_{\varepsilon})$.
 Since $F(0)=\ell u^*_\varepsilon v^*_\varepsilon(1-rs)>0$, $F(\gamma_{1,\varepsilon})<0$, $F(\gamma_{2,\varepsilon})<0$  and  $F(+\infty)>0$,  
$J_{4,\varepsilon}$ has two  distinct positive eigenvalues $\mu^+_{1,\varepsilon}\in (0,\gamma_{1,\varepsilon})$ and $\mu^+_{2,\varepsilon}\in(\gamma_{2,\varepsilon},+\infty)$. 
Let $\mu^+$ be $\mu^+_{1,\varepsilon}$ or $\mu^+_{2,\varepsilon}$, and $\varphi$ be the associated eigenvector.  
The first three equations of $(J_{4,\varepsilon}-\mu^+)\varphi=0$ leads to
\begin{eqnarray*}
-\mu ^{+}u+w_{1} &=& 0 \\
u_\varepsilon ^* u+(-c-\mu ^{+})w_{1}+(ru_\varepsilon ^* )v
&=& 0 \\
-\mu ^{+}v+w_{2} &=&0.
\end{eqnarray*}
Taking $u=-1$ in the above equations gives the components of $\varphi$, which is given by
\begin{align*}
\bar{u}& =-1, \\
\bar{w}_{1}& =-\mu ^{+}\\
\bar{v}& =\frac{1}{ru_\varepsilon ^* }\left( u_\varepsilon ^* -\left( c+\mu ^{+}\right) \mu ^{+}\right) 
=- \frac{1}{ru_\varepsilon ^* } q_\varepsilon(\mu^+)\\
\bar{w}_{2}& =\mu ^{+}\bar{v}.
\end{align*}
For the case $\mu ^{+}=\mu _{1,\varepsilon}^{+}$,  $q_\varepsilon\left( \mu^{+}\right) <0$, we have the corresponding  $\bar{v}>0$ and $\bar{w}_{2}>0$. For the case $\mu ^{+}=\mu _{2,\varepsilon}^{+}$,  $q_\varepsilon\left( \mu^{+}\right) >0$, we have  the corresponding $\bar{v}<0$, and  $\bar{w}_{2}<0$. This completes the proof.
\end{proof}

When $z<z^-$,  $\psi_\varepsilon(z)$ leaves $e_{4,\varepsilon}$  either along the strong unstable manifold of $e_{4,\varepsilon}$, which is tangent to the eigenvector   for $\mu^+_{2,\varepsilon}$, or along the weak unstable manifold of $e_{4,\varepsilon}$,  which is tangent to the eigenvector  for $\mu^+_{1,\varepsilon}$. We show that it must be the former case.

\begin{lemma}\label{lem:decreasing z<z-}
  $\psi_\varepsilon(z)$ leaves $e_{4,\varepsilon}$ along the strong unstable manifold of  $e_{4,\varepsilon}$, where both $u_\varepsilon(z)$ and $v_\varepsilon(z)$ are  monotone decreasing  for all $z<z^-$.  
\end{lemma}
\begin{proof}
The asymptotic form of  $\psi_\varepsilon(z)$ for $z<z^-$ is given by
\begin{align}\label{eq:asypm-form-at-e4}
(u_\varepsilon(z),w_{1,\varepsilon}(z),v_\varepsilon(z),w_{2,\varepsilon}(z))=(u^*_\varepsilon,0,v^*_\varepsilon,0)+Ce^{\mu^+ z}\varphi+o\left(e^{\mu^+ z}\right),\ \ \forall z<z^-,
\end{align}
where $\mu^+$ is either $\mu^+_{1,\varepsilon}$ or $\mu^+_{2,\varepsilon}$, and $\varphi$ is the associated eigenvector, and $C$ is a nonzero constant.
Recalling that $w_{1,\varepsilon}(z^-)<0$ and $w_{2,\varepsilon}(z^-)<0$, if $\mu=\mu^+_{1,\varepsilon}$,  in viewing that the sign of the eigenvector $\varphi$ for $\mu=\mu^+_{1,\varepsilon}$ is $(-,-,+,+)$, no nonzero constant $C$ can make the sign of \eqref{eq:asypm-form-at-e4} consistent. This implies  $\psi_\varepsilon(z)$ leaves $e_{4,\varepsilon}$ along the strong unstable manifold of  $e_{4,\varepsilon}$. Let $\mu=\mu^+_{2,\varepsilon}$, in viewing that  the sign of the eigenvector $\varphi$ for $\mu=\mu^+_{2,\varepsilon}$ is $(-,-,-,-)$, $C$ must be positive so that the sign of \eqref{eq:asypm-form-at-e4} is consistent. In this case, we have $w_{1,\varepsilon}(z)<0$ and $w_{2,\varepsilon}(z)<0$ for all $z<z^-$, which means both $u_\varepsilon(z)$ and $u_\varepsilon(z)$ are  monotone decreasing  for all $z<z^-$.
\end{proof}

The monotonicity and uniqueness of $(e_4,e_1)$-wave follow from the results of this section.  
\begin{theorem}\label{thm:monotone}
For each $\varepsilon\approx 0$, both the $u$-component  and $v$-component of the $(e_{4},e_{1})$-wave in Theorem \ref{thm:exist} are monotone decreasing. Moreover, the $(e_{4},e_{1})$-wave in Theorem \ref{thm:exist} is unique up to a shift in $z$.
\end{theorem}
\begin{proof}
The monotonicity follows from \eqref{eq:w_1,w_2<0 in cpt interval}, Lemma \ref{lem:decreasing z>z+} and \ref{lem:decreasing z<z-}. Since $\psi_\varepsilon(z)$ leaves $e_{4,\varepsilon}$ along the strong unstable manifold of $e_{4,\varepsilon}$ which is one dimensional, hence the $(e_{4},e_{1})$-wave leaves $e_{4,\varepsilon}$ only along one branch and then we obtain the uniqueness up to a shift in $z$.
\end{proof}

Here we note that in Chen et al. \cite{Chen-Hsiao-Wang},  the authors did not consider the asymptotic behavior of the $(e_4,e_1)$-wave as the wave approaches $e_4$. They prove the non-monotonicity by choosing some particular parameters such that the maximum value of the sub-solution is larger than the equilibrium $e_4$. Hence the relation between non-monotonicity and decay rates in the asymptotic behaviors is still not clear in [2]. In our case, we show that when the $(e_4,e_1)$-wave approaches $e_4$ along the strong unstable manifold, the corresponding wave is monotone.

\section{The essential spectrum}
 Let  $(U_\varepsilon ,V_\varepsilon ) $ be a  $(e_4$,$e_1)$-wave for $\varepsilon\approx 0$. The linearized operator around
the $\left( U_\varepsilon ,V_\varepsilon \right) $ is given by%
\begin{equation*}
L[\varphi ]=D_\varepsilon\partial _{z}^2\varphi +c\partial _{z}\varphi +B_\varepsilon\left(
z\right) \varphi,
\end{equation*}%
where%
\begin{equation}\label{phi D B}
\varphi =\left(
\begin{array}{c}
\phi_1 \\
\phi_2%
\end{array}%
\right) ,\ D_\varepsilon=\left(
\begin{array}{cc}
1 & 0 \\
0 & \varepsilon%
\end{array}%
\right) ,\ B_\varepsilon(z) =\left(
\begin{array}{cc}
1-2U_\varepsilon -rV_\varepsilon  & -rU_\varepsilon  \\
-s\ell V_\varepsilon  & \ell \left( 1-sU_\varepsilon -2V_{\varepsilon
}\right)%
\end{array}%
\right).
\end{equation}%
The function $\varphi$ belongs to the space $C_{0}$ defined in \eqref{def:C0}.
In this case the domain of definition $D(L)$ for $L$ is considered as the space $C_{0}^1$ of
all $\varphi \in C_{0}$ such that $\varphi $ and $\varphi '$ belong to
$C_{0}$. As $(U_\varepsilon,V_\varepsilon)$ is of monostable type,  the essential spectrum of $L$ acting in $C_0$ must have elements with positive real parts (see Henry \cite{Henry}, Volpert et al. \cite{Vvv}). As suggested in Volpert et al. \cite{Vvv},  we restrict $L$ to act in the weighted space $C_{\omega }$, defined in \eqref{def:C-omega}.
Denote by $\hat{L}_\omega$  the restriction of $L$ on $C_\omega$. 
Define the operator $L_\omega$ acting on $C_0$ by
\begin{equation*}
L_{\omega }[\varphi]:=\omega\hat{L}_{\omega
}[\omega^{-1}\varphi ],\ \varphi \in C_{0}%
.
\end{equation*}%
Then the spectrum of $\hat{L}_{\omega }$ (acting in $C_{\omega }$), coincides with
the spectrum of the operator $L_{\omega }$ (acting in $C_{0}$). The explicit form of $L_{\omega
}[\varphi ]$ is given by
\begin{equation}
L_{\omega }[\varphi ]=D_\varepsilon\partial _{z}^2\varphi +M_\varepsilon(z)\partial _{z}\varphi
+(R_\varepsilon(z)+B_\varepsilon(z))\varphi,  \label{gen op}
\end{equation}%
where%
\begin{align*}
M_\varepsilon(z) &=\diag\left(m_1(z),m^\varepsilon_2(z)\right):=cI-2g_1(z)D,\ \  
m_1(z)=c-2g_1(z) , \ m^\varepsilon_2(z) =c-2\varepsilon  g_1(z)\\
R_\varepsilon(z) &=\diag\left(r_1(z),r^\varepsilon_2(z)\right):=g_2(z)D-cg_1(z)I,\ \ 
r_1(z)=g_2-cg_1(z),\ r^\varepsilon_2(z)=\varepsilon g_2(z)-cg_1(z).
\end{align*}%
with
\begin{equation*}
g_1(z)=\frac{\delta e^{\delta z}}{1+e^{\delta z}},\ \ \ g_2(z)=\frac{%
\delta ^2e^{\delta z}\left( e^{\delta z}-1\right) }{\left( 1+e^{\delta
z}\right) ^2}.
\end{equation*}

 Let's denote $\sigma \left( L_\omega\right) $, $\sigma _{p}\left( L_\omega\right) $ and $\sigma
_{e}\left( L_\omega\right)$ the spectrum, point spectrum and essential spectrum of $L_\omega$, respectively. 
By  Henry \cite{Henry}, the essential spectrum $\sigma_{e}( L_\omega)$ lies on or contained in the region surrounded by the Fredholm borders, given by%
\begin{equation}\label{eq:Fredholm-S-pm}
S_\pm :=\left\{ \lambda \in \mathbb{C}\,\big|\,\det \left( -\tau ^2D_\varepsilon+i\tau M_\varepsilon^{\pm }+R_\varepsilon^{\pm}+B_\varepsilon^{\pm }-\lambda I \right)
=0,\tau \in \mathbb{R}\right\},
\end{equation}
where $M_\varepsilon^\pm=M_\varepsilon(\pm\infty)$, $R_\varepsilon^\pm=R_\varepsilon(\pm\infty)$, and $B_\varepsilon^\pm=B_\varepsilon(\pm\infty)$ are given by

\begin{equation}\label{eq:M-R-B-pm}
\begin{aligned}
&M_\varepsilon^+=cI-2\delta D_\varepsilon,& &R_\varepsilon^+=\delta^2 D_\varepsilon-c\delta I,& 
& B_\varepsilon^+=\left(\begin{matrix} 1 & 0 \\  0 & \ell \end{matrix}\right),&  \\
&M_\varepsilon^-=cI,& &R_\varepsilon^-=0,& 
& B_\varepsilon^-=\left(\begin{matrix} -u^*_\varepsilon & -r u^*_\varepsilon \\  -s\ell v^*_\varepsilon & -\ell v^*_\varepsilon \end{matrix}\right).&
\end{aligned}
\end{equation}

\begin{proposition} Assume $c>2$.  For  $\varepsilon\approx 0$, we have the following results for $\sigma_e(L_\omega)$.
\label{essential spectrum weighted}
\begin{enumerate}
\item If $\ell >2$ and $c^2\leq \frac{\ell ^2}{\ell -1}$, then  for all $ \delta>0$
\[ \sup\Re \sigma _e\left( L_{\omega }\right) >0.\]
In this case, the $(e_4$,$e_1$)-wave is unstable in $C_\omega$ for all $ \delta>0$. 
\item If $\ell \leq 2$, or $\ell >2$ and $\frac{\ell ^2}{\ell -1}<c^2$, 
we have that
\[ \sup\Re \sigma _e\left( L_{\omega }\right) <0\]
for all $\delta>0$ satisfying 
\begin{equation}
\delta ^2-c\delta +1<0\ \ \mbox{and}\ \  \varepsilon \delta ^2-c\delta +\ell <0.
\label{delta ineq weight}
\end{equation}
\end{enumerate}
\end{proposition}
\begin{proof} 
First, we show that $S_-$ lies in the left half of $\mathbb C$. Indeed, by \eqref{eq:Fredholm-S-pm}-\eqref{eq:M-R-B-pm}, the real part of $\lambda(\tau)\in S_-$ is given by
\[\Re \lambda(\tau)=\frac12\left(\alpha(\tau)\pm\sqrt{\Delta(\tau)}\right),\]
where
\[
\alpha(\tau)=-u^*_\varepsilon-\ell v^*_\varepsilon-(1+\varepsilon)\tau^2<0,\ 
\Delta(\tau)=\left(-u^*_\varepsilon+\ell v^*_\varepsilon-(1-\varepsilon)\tau^2\right)^2+4\ell r s u^*_\varepsilon v^*_\varepsilon.
\]
As $s=O(\varepsilon)$, we have 
\[ \sqrt{\Delta(\tau)}+\alpha(\tau) < -u^*_\varepsilon-\ell v^*_\varepsilon-(1-\varepsilon)\tau^2+
\left| -u^*_\varepsilon+\ell v^*_\varepsilon-(1-\varepsilon)\tau^2 \right|<0,\ \forall\tau\in\mathbb R,\] 
which implies $\sup\Re S_-<0$, for $\varepsilon\approx 0$. \\

\noindent On the other hand, since $M_\varepsilon^+$, $R_\varepsilon^+$ and $B_\varepsilon^+$ are diagonal, by \eqref{eq:Fredholm-S-pm}-\eqref{eq:M-R-B-pm}  the two branches of $\lambda(\tau)\in S_+$ are  given by
\[
\lambda(\tau)=-\tau^2+i\tau(c-2\delta)+\delta^2-c\delta+1,\ \mbox{ or }  -\varepsilon\tau^2+i\tau(c-2\varepsilon\delta)+\varepsilon\delta^2-c\delta+\ell.
\]
Taking the real parts of the above,
one can see that  for $\lambda(\tau)\in S_+$, 
 \[\max\limits_\tau \Re \lambda(\tau)=\max\{\delta ^2-c\delta +1, \varepsilon \delta ^2-c\delta +\ell \}.\]
Hence, $\sup\Re S_+<0$ for  $\delta>0$ satisfying \eqref{delta ineq weight}, 
i.e.,  $\delta\in (\delta_1,\delta_2)$, where $\delta_1$ and  $\delta_2$ are given by
\[\delta_1=\frac{c-\sqrt{c^2-4\varepsilon \ell }}{2\varepsilon }, \ \ \delta_2=\frac{c+\sqrt{c^2-4}}2. \]
Since 
\begin{equation*}
\delta_1=\frac{c-\sqrt{c^2-4\varepsilon \ell }}{2\varepsilon }=\frac{\ell }{c}+%
\frac{\varepsilon \ell ^2}{c^3}+O( \varepsilon ^2),
\end{equation*}%
we see that 
\[\begin{cases}
(\delta_1,\delta_2)\ne\ \emptyset,&  \mbox{if}\ \ell \leq 2; \mbox{ or }  \ell >2 \mbox{ and } \frac{\ell ^2}{\ell -1}<c^2,\smallskip\\
(\delta_1,\delta_2)=\ \emptyset, &  \mbox{if}\  \ell >2 \mbox{ and } c^2\leq \frac{\ell ^2}{\ell -1}.
\end{cases}
\] 
The proof is finished.
\end{proof}
\begin{remark}
The $(e_4$,$e_1$)-wave in Proposition \ref{essential spectrum weighted}(i) is also called absolutely unstable (see Sandstede and A. Scheel \cite{SandScheel}) which can be explained from the aspect of the absolute spectrum (see Kapitula and Promislow \cite{KapitulaPro}). Indeed, from the characteristic equation of $A_{+}\left( \lambda ,\varepsilon \right) $ in \eqref{char eqn of A+}, one can find that up to O($\varepsilon$) order, the absolute spectrum $\sigma _{\text{abs}}^{+}$ with respect to $A_{+}\left( \lambda ,\varepsilon \right)$ is
\begin{equation*}
\sigma _{\text{abs}}^{+}=\left\{ \lambda =\lambda _{1}+i\lambda
_{2}\left\vert \ \lambda _{2}^{2}=\frac{1}{c^{4}}\left( 2\left( \lambda
_{1}-\ell \right) +c^{2}\right) ^{2}\left( \lambda _{1}-\ell -c\sqrt{\ell -1}%
\right) \left( \lambda _{1}-\ell + c\sqrt{\ell -1}\right) \right. \right\} 
\end{equation*}
and the rightmost point of this curve is $\lambda _{1}=\ell -c\sqrt{\ell -1}$ which is nonnegative due to $\ell >2 \mbox{ and } c^2\leq \ell ^2(\ell -1)^{-1}$. Therefore $\sigma _{\text{abs}}^{+} \cap \left\{ \lambda :\text{Re}%
\lambda \geq 0\right\} \neq \emptyset $ and hence any weight $\delta>0$
cannot move the essential spectrum into the left half of $\mathbb C$.
\end{remark}

\section{Point spectrum}
In this section, we assume $\ell \leq 2$; or $\ell >2$ and $\ell ^2(\ell -1)^{-1}<c^2$ so that by Proposition \ref{essential spectrum weighted} the essential spectrum $\sigma _e\left( L_{\omega }\right)$ lies in the left half  of $\mathbb C$. Therefore we can select some small constant $k_0>0$ such that $\sup\Re \sigma _e\left( L_{\omega }\right) <-k_0$.
Throughout this section, let $\lambda\in \mathbb{C}_{k_0}:=\left \{ \lambda  |\  \Re\lambda > -k_{0}\right \}$.
The explicit form of the eigenvalue problem  $L_{\omega }\varphi =\lambda \varphi $
is given by
\begin{equation}\label{eq:eigen-problem-explicit}
\begin{aligned}
\phi_1^{\prime \prime }+m_1(z)\phi_1'+\left(
r_1(z)+b_{11}^{\varepsilon }(z)\right) \phi_1+b_{12}^{\varepsilon }(z)\phi
_2 &=\lambda \phi_1\\ 
\varepsilon \phi_2^{\prime \prime }+m_2^{\varepsilon }(z)\phi_2^{\prime
}+b_{21}^{\varepsilon }(z)\phi_1+\left( r_2^{\varepsilon
}(z)+b_{22}^{\varepsilon }(z)\right) \phi_2& =\lambda \phi_2,
\end{aligned}
\end{equation}
where $b^\varepsilon_{ij}(z)$ means the $(i,j)$ entry of $B_\varepsilon(z)$ in \eqref{phi D B}.\\

\noindent Let  $\psi=\partial _{z}\varphi $ and $y=\left( \varphi ,\psi \right)$  to write $L_{\omega}\varphi =\lambda \varphi $ 
as the first order system
\begin{equation}
\partial _{z}y=A\left( z;\lambda ,\varepsilon \right) y,
\label{y'=Ay}
\end{equation}%
where
\begin{equation*}
A\left( z;\lambda ,\varepsilon \right) =\left(
\begin{array}{cc}
0 & I \\
D_\varepsilon^{-1}\left( \lambda I-R_\varepsilon(z)-B_\varepsilon(z)\right) & -D_\varepsilon^{-1}M_\varepsilon(z)%
\end{array}%
\right).
\end{equation*}%
Let $A_{\pm }\left( \lambda ,\varepsilon \right) =A\left( \pm \infty;\lambda ,\varepsilon \right) $.

\begin{lemma} \label{number of matrix eigenvalues} 
For $\varepsilon \approx 0$ and $\lambda \in\mathbb{C}_{k_{0}}$, both the Morse index of $A_+\left( \lambda ,\varepsilon \right)$ and  $A_-\left( \lambda ,\varepsilon \right)$ are two.
\end{lemma}
\begin{proof}
The matrix eigenvalues of $A_+\left( \lambda ,\varepsilon \right) $ are roots of
\begin{equation}\label{char eqn of A+}
\mu^2+(c-2\delta)\mu+(\delta^2-c\delta)+1-\lambda=0, \ \mbox { or }\ \varepsilon\mu^2+(c-2\varepsilon\delta)\mu+(\delta^2\varepsilon-c\delta)+\ell-\lambda=0.
\end{equation}
For $\lambda\in \mathbb C_{k_0}\cap \mathbb R$ and $\lambda$ is large, both of the above equations have one positive root and one negative root, that is, the Morse index of $A_+\left( \lambda ,\varepsilon \right) $ is two.
According to Kapitula and Promislow \cite{KapitulaPro}, for $\lambda \in\mathbb{C}_{k_{0}}$ the Morse index of $A_+\left( \lambda ,\varepsilon \right) $ and $A_+\left( \lambda ,\varepsilon \right) $ are equal constant. This 
implies  that $A_-\left(
\lambda ,\varepsilon \right) $ and $A_+\left( \lambda ,\varepsilon \right) $ have the same Morse index of two for $\lambda \in\mathbb{C}_{k_{0}}$. 
\end{proof}

\subsection{Fundamental solutions}
Let's define the  linear spaces  $U_{\lambda }^-$ and  $S_{\lambda }^-$ of $C_0\times C_0$ for $\lambda\in \mathbb C_{k_0}$, by 
\begin{align*}
U_\lambda^- &:=\left \{  y(\cdot,\lambda)\in C_0\times C_0|\ y(z,\lambda) \mbox{ is a solution of \eqref{y'=Ay} with }\lim_{z\to
-\infty }y\left( z,\lambda\right) =0\right \},\\
S_\lambda^+ &:=\left \{ y(\cdot,\lambda)\in C_0\times C_0|\ y(z,\lambda) \mbox{ is a solution of \eqref{y'=Ay} with }\lim_{z\to
+\infty }y\left( z,\lambda \right) =0\right \}.
\end{align*}
Then, $\lambda \in \sigma _{p}\left( L_{\omega }\right) $ if and only if $%
U_{\lambda }^-\cap S_{\lambda }^+$ is non-trivial. By Lemma \ref{number
of matrix eigenvalues}, $\dim U_{\lambda }^-=\dim S_{\lambda }^+=2$, so  let 
$\left \{ y_1^-\left(z ,\lambda \right) ,y_2^-\left(z ,\lambda \right) \right \} $ be a basis of $U_{\lambda }^-$ and 
 $\left \{y_1^+\left( z,\lambda \right) ,y_2^+\left(z ,\lambda \right)\right \} $ be a basis of $S_{\lambda }^+$.  The Evans function is
defined by
\begin{equation*}
E\left( \lambda \right) =\det \left( y_1^-\left(
0,\lambda \right) ,y_2^-(0,\lambda),y_1^+\left(
0,\lambda \right) ,y_2^+(0,\lambda)\right),\ \lambda\in C_{k_0}.
\end{equation*}%
According to Kapitula and Promislow \cite{KapitulaPro},  $E(\lambda)$ is analytic and its zeros coincides with $\sigma_p(L_\omega)$, counting multiplicity.  We use the geometric singular perturbation theory to find  $y^\pm_i(z,\lambda)$. 
First, introducing the variable $\tau$ such that 
\begin{equation*}
z=z(\tau) =\frac1{2\kappa }\ln \left( \frac{1+\tau }{1-\tau }\right)
\end{equation*}
and write \eqref{y'=Ay} as the autonomous system for the augmented variable $Y=(y,\tau)$, given by
\begin{equation}
\left \{
\begin{array}{l}
\partial _{z}y=A\left( \tau ;\lambda ,\varepsilon \right) y\text{,\smallskip
} \\
\partial _{z}\tau =\kappa \left( 1-\tau ^2\right)\ (\tau(0)=0), %
\end{array}%
\right.  \label{y tau system}
\end{equation}%
where%
\begin{equation*}
A\left( \tau ;\lambda ,\varepsilon \right) =\left \{
\begin{array}{ll}
A\left( z(\tau) ;\lambda ,\varepsilon \right), \smallskip &
\text{for }\tau \neq \pm 1 \\
A_{\pm }\left( \lambda ,\varepsilon \right), & \text{for }\tau =\pm 1.%
\end{array}%
\right.
\end{equation*}
Here we also assume that $\kappa$ is sufficiently small such that the vector field \eqref{y tau system} to be $C^1$ (for the proof, see Alexander et al. \cite{AGJ}). Denote the equilibrium of   \eqref{y tau system} by $Y_{\pm 1}=(0,0,0,0,\pm 1)$.
Let's write  \eqref{y tau system}  as the first order system
\begin{equation}
\left \{
\begin{array}{l}
\phi_1'=\psi _1  \smallskip\\
\phi_2'=\psi _2  \smallskip \\
\psi _1'=\big(\lambda-r_1(\tau)-b^\varepsilon_{11}(\tau))\big)\phi_1-b^\varepsilon_{12}(\tau)\phi_2-m_1(\tau)\psi_1  \smallskip\\
\varepsilon \psi _2'=-b^\varepsilon_{21}(\tau)\phi_1+\big(\lambda-r^\varepsilon_2(\tau) -b^\varepsilon_{22}(\tau)\big)\phi_2-m^\varepsilon_2(\tau) \psi_2  \smallskip\\
\tau'=\kappa(1-\tau^2),
\end{array}
\right. 
\tag*{$(ES_\varepsilon)$} \label{slow eigen system}
\end{equation}
which is called the {\em slow eigenvalue system}.  
\noindent In terms of  the fast variable $\xi:=z/\varepsilon$,  using $\dot{(~)}=d(~)/d\xi$,  \ref{slow eigen system} is equivalent to  the {\em fast eigenvalue system} \ref{fast eigen system},
\begin{equation}\label{fast eigen system}
\left \{
\begin{array}{l}
\dot{\phi}_1=\varepsilon \psi _1  \smallskip\\
\dot{\phi}_2=\varepsilon \psi _2  \smallskip\\
\dot{\psi}_1
 =\varepsilon \left( (\lambda-r_1(\tau)-b^\varepsilon_{11}(\tau)) \phi_1-b^\varepsilon_{12}(\tau)\phi_2-m_1(\tau)\psi_1\right)\ \smallskip\\
\dot{\psi}_2
=-b^\varepsilon_{21}(\tau)\phi_1+\left(\lambda-r^\varepsilon_2(\tau) -b^\varepsilon_{22}(\tau)\right)\phi_2-m^\varepsilon_2(\tau) \psi_2  \smallskip\\
\dot{\tau}=\varepsilon \kappa \left( 1-\tau ^2\right).%
\end{array}%
\right. 
\tag*{$(EF_\varepsilon)$} 
\end{equation}%

\noindent Taking $\varepsilon=0$ (then $b_{21}^0(\tau)=0$ by $s=O(\varepsilon)$,  $r^0_2=-cg_1(\tau)$ and $m^0_2=c$) for \ref{slow eigen system}  leads to  the {\em reduced eigenvalue system} \ref{slow reduced}%
\begin{equation}
\left \{
\begin{array}{l}
\phi_1'=\psi _1  \smallskip\\
\phi_2'=c^{-1}a_2^0(\tau,\lambda)\phi_2 \smallskip\\
\psi _1'=a_1^0(\tau,\lambda)\phi_1-b^0_{12}(\tau)\phi_2-m_1(\tau)\psi_1  \smallskip\\
0=a_2^0(\tau,\lambda)\phi_2-c \psi_2  \smallskip\\
\tau '=\kappa \left( 1-\tau ^2\right) ,%
\end{array}%
\right. 
\tag*{$(ES_0)$} \label{slow reduced}
\end{equation}
and \noindent taking $\varepsilon=0$ for \ref{fast eigen system}  leads to  the {\em layer eignvalue system} \ref{fast reduced}
\begin{equation}
\left \{
\begin{array}{l}
\dot{\phi}_1=0\\\smallskip
\dot{\phi}_2=0 \\\smallskip
\dot{\psi}_1=0 \\\smallskip
\dot{\psi}_2=a_2^0(\tau,\lambda) \phi_2-c \psi _2 \\\smallskip
\dot{\tau}=0,%
\end{array}%
\right. 
\tag*{$(EF_0)$}  \label{fast reduced}
\end{equation}%
where
\begin{align*}
a^\varepsilon_1(\tau,\lambda) &=\lambda-r_1(\tau)-b^\varepsilon_{11}(\tau) \\
a^\varepsilon_2(\tau,\lambda) &=\lambda-r^\varepsilon_2(\tau) -b^\varepsilon_{22}(\tau).
\end{align*}
 The reduced eigenvalue system \ref{slow reduced} has  a 4-dimensional slow manifold $\mathscr M$, given by
\begin{equation}
\mathscr M=\left \{ \left( \phi_1,\phi_2,\psi _1,\psi _2,\tau
\right) \left |\  \psi _2= H(\phi_2,\tau):=c^{-1}a^0_2 (\tau,\lambda) \phi_2,\ 
\phi_1,\phi_2,\psi _1\in \mathbb{C}
,\tau \in [-1,1]\right. \right \},  \label{M_0}
\end{equation}%
which is also  the equilibrium of \ref{fast reduced}. The linearization of \ref{fast reduced} around $(\phi_1,\phi_2,\psi_1,\psi_2,\tau)\in\mathscr M$ is given by
\begin{equation}
\left(
\begin{array}{ccccc}
0 & 0 & 0 & 0 & 0 \\
0 & 0 & 0 & 0 & 0 \\
0 & 0 & 0 & 0 & 0 \smallskip\\
0 & a^0_2(\tau,\lambda)  & 0 &-c  & \partial _{\tau }a^0_2(\tau,\lambda) \phi_2  \smallskip \\
0 & 0 & 0 & 0 & 0%
\end{array}%
\right),  \label{Jacobian at p of fast slow system}
\end{equation}%
which has  eigenvalue $0$ with multiplicity four  (corresponding to the tangential space of $\mathscr M$), and $-c<0 $. Hence,  $\mathscr M$ is normally hyperbolic, which implies that \ref{slow eigen system} has a normally hyperbolic invariant manifold $\mathscr M_\varepsilon $ for $\varepsilon\approx 0$, which is $O(\varepsilon)$-closed to $\mathscr M$, given by 
\begin{equation}
\mathscr M_\varepsilon=\big \{ \left( \phi_1,\phi_2,\psi _1,\psi _2,\tau
\right) \left |\  \psi _2= H_\varepsilon(\phi_1,\phi_2,\psi_1,\tau),\ 
\phi_1,\phi_2,\psi _1\in \mathbb{C}
,\tau \in [-1,1]\right. \big \},
\end{equation}
where 
\[ H_\varepsilon\left(\phi_1,\phi_2,\psi_1,\tau\right)=H(\phi_2,\tau)+O(\varepsilon).\]

For the flow  of  \ref{slow reduced} restricted on $\mathscr M$,  let's denote $\mathscr W^s(Y_{+1})$  the stable  manifold  of $Y_{+1}$,  and denote $\mathscr W^u(Y_{-1})$ the  unstable  manifold  of $Y_{-1}$.  It's clear that both $\mathscr W^s(Y_{+1})$ and $\mathscr W^u(Y_{-1})$ contain the 1-dimensional submanifold
\[
\mathscr T:=\{(0,0,0,0)\}\times \{\tau\in(-1,1)\},
\]
which corresponds to the trivial solutions of  \eqref{y'=Ay}.

  \begin{lemma}\label{lemma:fund-sol-on-slow-0}
Let $\lambda\in C_{k_0}$. 
$\mathscr W^u(Y_{-1})$  is 3-dimensional, and  $\mathscr W^s(Y_{+1})$ of is 2-dimensional. Furthermore, 
 \begin{enumerate}
 \item $\mathscr W^u(Y_{-1})$ has solutions $Y^-(z,\lambda)$ and $\hat Y^-(z,\lambda)$ to \ref{slow reduced}, of the from 
 \begin{align*}
 Y^-(z,\lambda)&=(\phi^-_1(z,\lambda),0,\psi^-_1(z,\lambda),0,\tau(z) )\\
 \hat Y^-(z,\lambda)&=(\hat\phi^-_1(z,\lambda),\hat\phi^-_2(z,\lambda),\hat\psi^-_1(z,\lambda),\hat\psi^-_2(z,\lambda),\tau(z)),
 \end{align*} 
 where $(\hat\phi^-_2(z,\lambda),\hat\psi^-_2(z,\lambda))$ is nontrivial;
\item  $\mathscr W^s(Y_{+1})$ has a solution $Y^+(z,\lambda)$  to \ref{slow reduced}, of the form
\[ Y^+(z,\lambda)=(\phi^+_1,0,\psi^+_1,0,\tau(z)).\] 
 \end{enumerate}

 \end{lemma}
\begin{proof} 
The flow of \ref{slow reduced} restricted on $\mathscr M$ is governed by the equation for $(\phi_1,\phi_2,\psi_1,\tau)$, given by
\begin{equation}\label{eq:slow-reduced}
\left \{
\begin{array}{l}
\phi_1'=\psi _1\smallskip\\
\psi _1'=a_1^0(\tau,\lambda) \phi_1-b^0_{12}(\tau)\phi_2-m_1(\tau)\psi _1\smallskip\\
\phi_2'=c^{-1}a_2^0(\tau,\lambda)\phi_2\smallskip\\
\tau '=\kappa \left( 1-\tau ^2\right),%
\end{array}%
\right. 
\end{equation}
whose linearization is the upper triangular block matrix $Q(\tau,\lambda)$ given by
\[ Q(\tau,\lambda):=\left(
\begin{matrix} Q_1(\tau,\lambda)_{2\times 2} &*_{2\times 1}  & *_{2\times 1} \smallskip\\ 
0_{1\times 2} & c^{-1}a^0_2(\tau,\lambda) &  *_{1\times 1} \smallskip\\
0_{1\times 2} & 0_{1\times 1} &-2k\tau \end{matrix}\right),\ \ \mbox{where}\ 
Q_1(\tau,\lambda):=\left(
\begin{matrix}
0 & 1  \smallskip\\ 
a_1^0(\tau,\lambda)&- m_1(\tau)  
\end{matrix}
\right).
\]
 At $\tau= +1$, the eigenvalues of  $Q(+1,\lambda)$ consist of $-2\kappa<0$, $c^{-1}a^0_2(+1,\lambda)$ (with 
 \begin{equation*}
     \Re c^{-1}a^0_2(+1,\lambda)= c^{-1}(\Re \lambda+c\delta-\ell)>0
 \end{equation*}
by \eqref{delta ineq weight} and $\lambda\in C_{k_0}$),  and the eigenvalues of $Q_1(+1,\lambda)$. As  $Q_1(+1,\lambda)$ is given by
\[Q_1(+1,\lambda)=\left(\begin{matrix}
 0  &1  \\
 \lambda-(\delta^2-c\delta+1) & -(c-2\delta) 
\end{matrix}\right), 
 \]
its eigenvalues  $\nu$ satisfies
\[ \nu^2+(c-2\delta)\nu-[\lambda-(\delta^2-c\delta+1)]=0.\] 
Extracting the equation for $\Re\nu$  gives
\[ (\Re\nu)^2+(c-2 \delta)(\Re\nu)-[(\Im\nu)^2+\Re\lambda-(\delta^2-c\delta+1)]=0,\]
which has one positive root and one negative root since $\Re\lambda >\delta^2-c\delta+1$ by \eqref{delta ineq weight} and $\lambda\in C_{k_0}$. 
It follows that  $Q(+1,\lambda)$ has two eigenvalues with positive real parts and two eigenvalues with negative real parts, which implies the stable manifold $\mathscr W^s(Y_{+1})$ is 2-dimensional. \\

\noindent At $\tau= -1$, the eigenvalues of $Q(-1,\lambda)$ consist of $2\kappa>0$, $c^{-1}a^0_2(-1,\lambda)$ (with 
\begin{equation*}
    \Re c^{-1}a^0_2(-1,\lambda)= c^{-1}(\Re \lambda+\ell )>0
\end{equation*}
by \eqref{delta ineq weight} and $\lambda\in C_{k_0}$) and the eigenvalues of $Q_1(-1,\lambda)$. As $Q_1(-1,\lambda)$ is given by 
\[ Q_1(-1,\lambda)=\left(\begin{matrix}
 0  &1  \\
 \lambda+u_0^* & -c 
\end{matrix}\right),
 \]
its eigenvalues  $\mu$ satisfying 
\[ \mu^2+c\mu-(\lambda+u_0^*)=0.\]
Extracting the equation for $\Re \mu$  gives
\[ (\Re\mu)^2+c\Re\mu-\left((\Im\mu)^2+\Re\lambda+u_0^*\right)=0,\]
which has one positive root and one negative root for $\lambda\in C_{k_0}$. It follows that $Q(-1,\lambda)$ has three eigenvalues with positive real parts and one eigenvalue with negative real part, which implies  the unstable manifold  $\mathscr W^u(Y_{-1})$ is  3-dimensional.\\

Obviously, the submanifold $\mathscr M_1:=\mathscr M\cap\{\phi_2=\psi_2=0\}$  of $\mathscr M$ is invariant under \ref{slow reduced}.  The flow of \ref{slow reduced} restricted on $\mathscr M_1$ is governed by the equation for $(\phi_1,\psi_1,\tau)$, given by
\begin{equation}\label{eq:slow-reduced-1}
\left \{
\begin{array}{l}
\phi_1'=\psi _1\smallskip\\
\psi _1'=a_1^0(\tau,\lambda) \phi_1-m_1(\tau)\psi _1\smallskip\\
\tau '=\kappa \left( 1-\tau ^2\right),%
\end{array}%
\right. 
\end{equation}
whose linearization of the first two equations is exactly $Q_1(\tau,\lambda)$.  As $Q_1(-1,\lambda)$ has one eigenvalue with positive real part and one eigenvalue with negative real part,  there is a solution $Y^-\in\mathscr W^u(Y_{-1})\cap \mathscr M_1 $, whose  $(\phi_2,\psi_2)$  component is trivial. Since $\mathscr W^u(Y_{-1})$ is 3-dimensional, it must possess a solution $\hat Y^-\in\mathscr M\backslash (\mathscr M_1\cup\mathscr T)$, whose  $(\phi_2,\psi_2)$  component is nontrivial. \\

Similarly, as $Q_1(+1,\lambda)$ has one eigenvalue with positive real part and one eigenvalue with negative real part,  there is a solution $Y^+\in\mathscr W^s(Y_{+1})\cap \mathscr M_1 $, whose  $(\phi_2,\psi_2)$  component is trivial.
\end{proof}

For the flow  of  \ref{slow eigen system} restricted on $\mathscr M_\varepsilon$,  let's denote $\mathscr W^s_\varepsilon(Y_{+1})$  the {\em local stable manifold  of $Y_{+1}$},  and denote $\mathscr W^u_\varepsilon(Y_{-1})$ the  {\em local unstable  manifold}  of $Y_{-1}$.  Note that the invariant manifold $\mathscr M_\varepsilon$ for \ref{slow eigen system} also contains $Y_{\pm}$ as its equilibria.  
\begin{lemma}\label{lemma:fund-solu-Y-on-slow}
Let $\lambda\in C_{k_0}$. For the flow  of  \ref{slow eigen system} restricted on $\mathscr M_\varepsilon$, we have
 \begin{enumerate}
 \item $\mathscr W^u_\varepsilon(Y_{-1})$ has solutions $Y^-_\varepsilon$ and $\hat Y^-_\varepsilon$ that are $O(\varepsilon)$-closed to $Y^-$ and $\hat Y^-$, respectively, 
\[Y^-_\varepsilon(z,\lambda)=Y^-(z,\lambda)+O(\varepsilon), \ \ \ \hat Y^-_\varepsilon(z,\lambda)=\hat Y^-(z,\lambda)+O(\varepsilon);\] \item  $\mathscr W^s_\varepsilon(Y_{+1})$ has a solution $Y^+_\varepsilon(z,\lambda)$ which is $O(\varepsilon)$-closed to $Y^+$,
\[Y^+_\varepsilon(z,\lambda)=Y^+(z,\lambda)+O(\varepsilon).\] 
 \end{enumerate}
\end{lemma}
\begin{proof}
 According to  Fenichel \cite{Feni}, Szmolyan \cite{Sz} and Lemma \ref{lemma:fund-sol-on-slow-0}, 
for $\varepsilon \approx 0$, $\mathscr W^u_\varepsilon(Y_{-1})$  has orbits  $O(\varepsilon)$-closed to $Y^-$, $\hat Y^-\in \mathscr W^u(Y_{-1})$, respectively. Similarly, $\mathscr W^s_\varepsilon(Y_{+1})$  has an orbit  $O(\varepsilon)$-closed to $Y^+\in \mathscr W^s(Y_{+1})$.
\end{proof}

 Now we have two fundamental solutions in $U_{\lambda }^-$,  which come from the $(\phi_1.\phi_2,\psi_1,\psi_2)$ component of 
 $Y^-_\varepsilon$ and $\hat Y^-_\varepsilon$, given by
 \[ y^-_1(z,\lambda)=
 \left(\begin{matrix} \phi^-_1(z,\lambda)+O(\varepsilon) \\ O(\varepsilon)  \\ 
 \psi^-_1(z,\lambda)+O(\varepsilon)  \\ O(\varepsilon)  \\\end{matrix}\right),\ \ 
y^-_2(z,\lambda)=
 \left(\begin{matrix} \hat\phi^-_1(z,\lambda)+O(\varepsilon) \\ \hat\phi^-_2(z,\lambda)+O(\varepsilon) 
  \\ \hat\psi^-_1(z,\lambda)+O(\varepsilon)  \\ \hat\psi^-_2(z,\lambda)+O(\varepsilon)  \\\end{matrix}\right),  
 \]
 and one  fundamental solution in $S_{\lambda }^+$,  which comes from the $(\phi_1.\phi_2,\psi_1,\psi_2)$ component of 
 $Y^+_\varepsilon$, given by
 \[ y^+_1(z,\lambda)=
 \left(\begin{matrix} \phi^+_1(z,\lambda)+O(\varepsilon) \\ O(\varepsilon)  \\ 
 \psi^+_1(z,\lambda)+O(\varepsilon)  \\ O(\varepsilon)  \\\end{matrix}\right).\ \  
 \]

\medskip
\noindent We need to find another fundamental solution in $S^+_\lambda$.  
\begin{lemma}\label{lem:tilde-Y+} Let $\lambda\in C_{k_0}$. 
For $\varepsilon\approx 0$, \ref{slow eigen system} has a fundamental solution 
\begin{equation*}
\tilde Y^+_\varepsilon (z,\lambda)=\left(\tilde\phi^+_1(z,\lambda), \tilde\phi^+_2(z,\lambda), \tilde\psi^+_1(z,\lambda), \tilde\psi^+_2(z,\lambda),\tau(z)\right)
\end{equation*}
satisfying $\tilde Y^+_\varepsilon (+\infty,\lambda)=Y_{+1}$ and 
\begin{align}\label{eq:tilde-psi2+}
 \tilde\psi^+_2(0,\lambda)-H\left(\tilde\phi_2^+(0,\lambda), 0\right)=O(1).
\end{align}
\end{lemma}
\begin{proof}  To simplify the notations, let's denote $Y_s=(\phi_1,\phi_2,\psi_1)$. 
The orbit $(Y_s(\xi,\lambda),\psi_2(\xi,\lambda),\tau(\xi))$ of \ref{fast reduced} starting at $\left(Y^{(0)}_s, \psi^{(0)}_2,0\right)$
 is given by
\[
\left(Y^{(0)}_s,\psi_2(\xi,\lambda),0\right),\ \ \mbox{where}\ \lim_{\xi\to+\infty}\psi_2(\xi,\lambda)
=H\left(\phi_2^{(0)},0\right).
\]
This implies that the stable manifold  $\mathscr N^s_{+1}$ of $\mathscr W^s(Y_{+1})$ for  \ref{fast reduced} is foliated by the points in $\mathscr W^s(Y_{+1})$, 
\[
\bigcup_{p\in \mathscr W^s(Y_{+1})} \mathscr F^{s}(p),\] 
where  $p=\left(Y^{(0)}_s,H\left(\phi_2^{(0)},0\right),0\right)\ \in \mathscr W^s(Y_{+1})$,  associated with $\mathscr F^{s}(p)$  given by
\[
\mathscr F^{s}(p)=\left \{(Y^{(0)}_s,\psi_2,0),\ \psi_2\in\mathbb R\right \},
\]
as shown in Figure \ref{fig:5}.
 Let $\mathscr N^s_{+1,\varepsilon}$ denote the  {\em local stable manifold} of $Y_{+1}$ for \ref{slow eigen system}. Then $\mathscr N^s_{1,\varepsilon}$  is foliated by  the points in $\mathscr W^s_\varepsilon(Y_{+1})$, 
\[
\bigcup_{p\in \mathscr W^s_\varepsilon(Y_{+1})} \mathscr F^{s}_\varepsilon(p),\] 
as shown in Figure \ref{fig:6}. This implies that  $\mathscr N^s_{+1,\varepsilon}$  has an orbit  which does not belong to  $\mathscr M_\varepsilon$, denoted by 
\[\tilde Y^+_\varepsilon(z,\lambda)=\big(\tilde\phi^+_1(z,\lambda), \tilde\phi^+_2(z,\lambda), 
 \tilde\psi^+_1(z,\lambda), \tilde\psi^+_2(z,\lambda),\tau(z)\big), 
\]
where $\tilde Y^+_\varepsilon(z,\lambda)$ is away form $\mathscr M_\varepsilon$ at some $z_0$.
Without loss of generality, we may assume $z_0=0$, which implies
\[ \tilde\psi^+_2(0,\lambda)-H_\varepsilon\left(\tilde\phi^+_1(0,\lambda),\tilde\phi^+_2(0,\lambda),\tilde\psi^+_1(0,\lambda), 0\right)=O(1).\]
As $H_\varepsilon\left(\phi_1,\phi_2,\psi_1,\tau\right)=H(\phi_2,\tau)+O(\varepsilon)$, we have 
\[ \tilde\psi_2^+(0,\lambda)-H\left(\tilde\phi^+_2(0,\lambda),0\right)=O(1).\]
\end{proof}

Taking  the $(\phi_1,\phi_2,\psi_1,\psi_2)$ component of 
 $\tilde Y^+_\varepsilon$ gives another  fundamental solution in $S_{\lambda }^+$, given by
 \[ y^+_2(z,\lambda)=
 \left(\begin{matrix} \tilde\phi^+_1(z,\lambda) \smallskip\\ \tilde\phi^+_2(z,\lambda)   \smallskip\\ 
 \tilde\psi^+_1(z,\lambda) \smallskip \\ \tilde\psi^+_2(z,\lambda)  \\\end{matrix}\right).\ \  
 \]
 \begin{figure} 
\begin{subfigure} [b]{0.5\textwidth} 
\centering\includegraphics[scale=0.5]{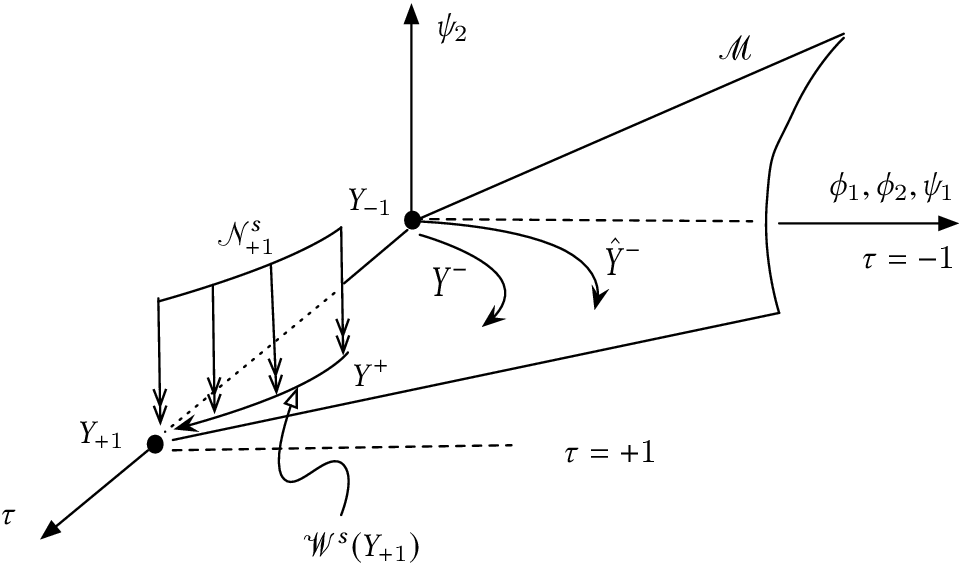} 
\subcaption{$\varepsilon=0$} \label{fig:5} 
\end{subfigure} 
\begin{subfigure} [b]{0.5\textwidth} 
\centering\includegraphics[scale=0.5]{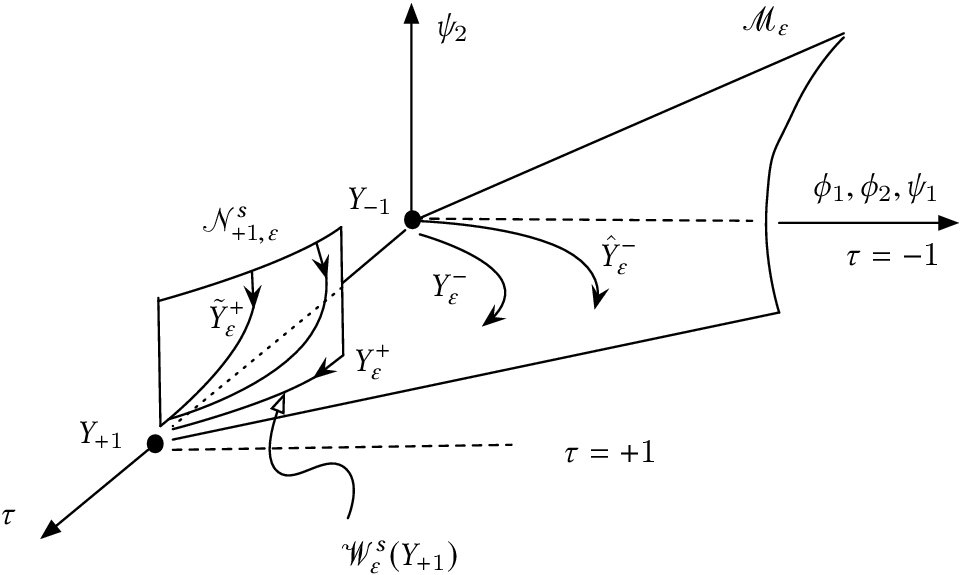} 
\subcaption{$\varepsilon>0$} \label{fig:6} 
\end{subfigure} 
\caption{}
\end{figure} 

\subsection{Stability}
From the above study, for $\varepsilon\approx 0$ the Evans function can be defined as
\begin{align*}
E\left( \lambda \right) & =\det \left(y^-_1(0,\lambda), y^-_2(0,\lambda),y^+_1(0,\lambda),y^+_2(0,\lambda) \right)
=E_0(\lambda)+O(\varepsilon), 
\end{align*}
where
\begin{eqnarray*}
E_{0}\left( \lambda \right)  &=&\det \left( 
\begin{array}{cccc}
\phi_1^-(0,\lambda) & \hat{\phi}_1^-(0,\lambda)  & \phi_1^+(0,\lambda) &  \tilde\phi^+_1(0,\lambda) \\ 
0 & \hat\phi_2^-(0,\lambda) & 0 &  \tilde\phi^+_2(0,\lambda) \\ 
\psi^-_1(0,\lambda) & \hat\psi^-_1(0,\lambda) & \psi_1^+(0,\lambda) &  \tilde\psi^+_1(0,\lambda) \\ 
0 & \hat\psi_2^-(0,\lambda) & 0 &  \tilde\psi^+_2(0,\lambda)%
\end{array}%
\right)  \\
&=&\det \left( 
\begin{array}{cc}
\phi_1^-(0,\lambda) & \phi_1^+(0,\lambda)  \\ 
\psi_1^-(0,\lambda) &\psi_1^+(0,\lambda)
\end{array}%
\right)
\det\left( 
\begin{array}{cc}
\hat\phi^-_2(0,\lambda) & \tilde\phi_2^+(0,\lambda)  \\ 
\hat\psi^-_2(0,\lambda) &\tilde\psi_2^+(0,\lambda)
\end{array}%
\right).
\end{eqnarray*}
Since  $\hat\psi^-_2(0,\lambda)=H\left(\hat\phi^-_2(0,\lambda),0\right)=c^{-1}a_2^0(0,\lambda)\hat\phi_2^-(0,\lambda)$, we have 
\begin{align*}
\det\left( 
\begin{array}{cc}
\hat\phi^-_2(0,\lambda) & \tilde\phi_2^+(0,\lambda)  \\ 
\hat\psi^-_2(0,\lambda) &\tilde\psi_2^+(0,\lambda)
\end{array}
\right)
&=\hat\phi^-_2(0,\lambda)\tilde\psi_2^+(0,\lambda)-c^{-1}a_2^0(0,\lambda)\hat\phi_2^-(0,\lambda)\tilde\phi^+_2(0,\lambda) \  \\
&=\hat\phi^-_2(0,\lambda)\left(\tilde\psi_2^+(0,\lambda)-H\left(\tilde\phi^+_2(0,\lambda),0\right)\right) \\
&=\hat\phi^-_2(0,\lambda) O(1).
\end{align*}
If $\hat\phi^-_2(0,\lambda)=0$, then we have $\hat\psi^-_2(0,\lambda)=H(\hat\phi^-_2(0,\lambda),0)=0$ because of $\hat Y^-(z,\lambda)\in\mathscr M$, which contradicts to that the submanifold $\mathscr M\cap\{\phi_2=\psi_2=0\}$ is invariant and that the $(\phi_2,\psi_2)$ component of $\hat Y^-(z,\lambda)\in\mathscr M$ is nontrivial. Hence we have $\hat\phi^-_2(0,\lambda)\ne 0$. 
This implies that the zeros of $E_{0}\left( \lambda \right) $ coincide with that of $E_1(\lambda)$, defined by
\begin{equation*}
E_1\left( \lambda \right) :=\det \left( 
\begin{array}{cc}
\phi_1^-(0,\lambda) & \phi_1^+(0,\lambda  \\ 
\psi_1^-(0,\lambda) & \psi_1^+(0,\lambda)
\end{array}%
\right). 
\end{equation*}%
In viewing that $\phi_1^-$ and  $\phi_1^+$ are eigenfunctions of  the eigenvalue problem
\begin{equation}\label{L1 phi_1=lambda phi_1}
L_1 \phi_1:=\phi_1^{\prime \prime }+m_1(z) \phi
_1'+q(z) \phi_1=  \lambda \phi_1,
\end{equation}
with
\begin{equation*}
q(z) =r_1(z)+b_{11}^0(z) =g_2(z) -cg_1(z) +1-2U_{0}(z)-rV_{0}(z),
\end{equation*}
which is obtained by letting $\phi_2=0$ in $L_{\omega }\varphi =\lambda\varphi $ with $\varepsilon =0$.
$E_1\left( \lambda \right)$ is exactly the Evans function for  the operator $L_1$.
\begin{lemma}\label{eigenvalue of L1 <=0}
If $\lambda $ is an eigenvalue of $L_1$, then $\lambda\in\mathbb R$ and $\lambda < 0$.
\end{lemma}
\begin{proof}
Suppose that $\lambda\geq 0$ is an eigenvalue of $L_1$ associated with the eigenfunction $\phi_1$. Let 
\begin{equation*}
\phi_3=\phi_1 e^{\frac{1}{2} \int m_1(z)dz}
\end{equation*}
to transform  \eqref{L1 phi_1=lambda phi_1} into 
\begin{equation}
\phi_3^{\prime \prime }+\left( q(z) -\frac{1}{4}m_1^2(z)+g'_1(z) \right) \phi_3=\lambda \phi_3.
\label{Phi_3 star}
\end{equation}%
As the left-hand side operator for $\phi_3$ is self-adjoint, it follows that if $\lambda$ is an eigenvalue of $L_1$, then $\lambda\in\mathbb R$. Multiplying  both sides of \eqref{Phi_3 star} with $\phi_3$, and  integrating from $-\infty $ to $+\infty $, then using integration by
parts gives
\begin{equation}
\int_{-\infty }^{+\infty }\left( q(z) -\frac1{4}m_1^2(z) +g_1'(z) \right) \left[ \phi_3(z) \right] ^2dz=\lambda\int_{-\infty }^{+\infty }%
\left[ \phi_3(z) \right] ^2 dz+\int_{-\infty}^{+\infty }\left[ \phi_3'(z) \right] ^2dz>0.  \label{Phi_3 integral >0}
\end{equation}
However, using $g'_1=g^2_1-g_2$ to simplifying the left side of \eqref{Phi_3 integral >0} leads to%
\begin{eqnarray*}
q(z) -\frac1{4}m_1^2(z) +g_1'(z) 
= \frac{4-c^2}{4}-2U_{0}(z) -rV_{0}\left(z\right) <0,
\end{eqnarray*}%
because of $c>2$ and both $U_{0}(z) $ and $V_{0}(z) >0$,
for all $z\in \mathbb{R}$. This contradicts to \eqref{Phi_3 integral >0}. Hence $0>\lambda\in\mathbb R$.
\end{proof}

\begin{proposition}
All zeros of $E_1(\lambda)$ are real and negative.  
\end{proposition}

\medskip
\noindent
Concluding the above arguments, we obtain the stability results for  $\left( e_{4},e_1\right) $-waves. 
\begin{theorem}\label{thm:stability} 
Assume $r<1$ and $s=O\left( \varepsilon \right) $ in \eqref{eq:LV}. For the $\left( e_{4},e_1\right) $-waves  obtained in Theorem \ref{thm:exist},
\begin{enumerate}
\item if $\ell >2$ and $4<c^2\leq \frac{\ell ^2}{\ell -1}$, then the $\left( e_{4},e_1\right) $-wave is unstable in the space $C_{\omega }$ for all $\delta>0$;
\item if $\ell \leq 2$ and $c>2$; or $\ell >2$ and $\frac{\ell ^2}{\ell -1}<c^2$, 
  then the $\left( e_{4},e_1\right) $-wave is asymptotically stable with no shift in the space $C_{\omega }$ for all $\delta$  satisfying \eqref{delta ineq weight}. 
 \end{enumerate}
\end{theorem}

Finally, we provide some discussion on future work. When $d=O(1)$ or $s=O(1)$ in \eqref{eq:LV-0}, analyzing the stability and monotonicity of $(e_4,e_1)$-waves becomes a challenging problem, which needs new ideas and techniques.

\medskip
\medskip


\noindent{\bf Funding}
C.H. Chang is supported by the National Science and Technology Council, Taiwan (Grant numbers 114-2115-M-194-002), T. S. Yang is supported by the National Science and Technology Council, Taiwan (Grant numbers 113-2115-M-029-001).

\end{document}